\documentclass[11pt]{amsart}
\usepackage{graphicx, color}
\usepackage{amscd}
\usepackage{amsmath}
\usepackage{amsfonts}
\usepackage{amssymb}
\usepackage{mathrsfs}
\newtheorem{theorem}{Theorem}
\newtheorem{lemma}[theorem]{Lemma}
\newtheorem{proposition}[theorem]{Proposition}

\theoremstyle{definition}
\newtheorem{remark}[theorem]{Remark}
\numberwithin{equation}{section}
\newcommand{\RR}{\mathbb R}
\newcommand{\NN}{\mathbb N}

\newcommand{\ve}{\varepsilon}

\begin{document}

%\hfill\today\bigskip

\title[On multiple solutions via $\nabla$-theorems]{On multiple solutions for nonlocal fractional problems via $\nabla$-theorems}

\thanks{The first and the third author were supported by the INdAM-GNAMPA Project 2015 {\it Modelli ed equazioni non-local di tipo frazionario}. The second author is a member of GNAMPA and is supported by the MIUR National Research Project {\it Variational and perturbative aspects of nonlinear differential problems}. The third author was supported by the MIUR National Research Project {\it Variational and Topological Methods in the Study of Nonlinear
Phenomena} and by the ERC grant $\epsilon$ ({\it Elliptic Pde's and
Symmetry of Interfaces and Layers for Odd Nonlinearities}).}

%%%%%%%%%%%%%%%%%%%%%%%%%%%%%%%%%%%%%%%%%%%%%%%%%%%%%%%%%%%%%%%%%%%%%%%
\author[G. Molica Bisci]{Giovanni Molica Bisci}
\address{Dipartimento PAU,
          Universit\`a `Mediterranea' di Reggio Calabria,
          Via Melissari 24, 89124 Reggio Calabria, Italy}
\email{\tt gmolica@unirc.it}

\author[D. Mugnai]{Dimitri Mugnai}
\address{Dipartimento di Matematica e Informatica,
          Universit\`a degli Studi di Perugia,
          Via Vanvitelli, 1 06123 Perugia, Italy}
\email{\tt dimitri.mugnai@unipg.it}

\author[R. Servadei]{Raffaella Servadei}
\address{Dipartimento di Scienze di Base e Fondamenti (DiSBeF),
 Universit\`a degli Studi di Urbino `Carlo Bo',
 Piazza della Repubblica, 13
 61029 Urbino (Pesaro e Urbino)
 ITALY }
\email{\tt raffaella.servadei@uniurb.it}

\keywords{Fractional Laplacian, variational methods, $\nabla$-theorems, $\nabla$-condition, superlinear and subcritical nonlinearities.\\
\phantom{aa} 2010 AMS Subject Classification: Primary: 35J20, 35S15;
Secondary: 47G20, 45G05.}

%%%%%%%%%%%%%%%%%%%%%%%%%%%%%%%%%%%%%%%%%%%%%%%%%%%%%%%%%%%%%%%%%%%%%%%%

\begin{abstract}
The aim of this paper is to prove multiplicity of solutions for
nonlocal fractional equations modeled by
$$ \left\{
\begin{array}{ll}
(-\Delta)^s u-\lambda u=f(x,u) & {\mbox{ in }} \Omega\\
u=0 & {\mbox{ in }} \RR^n\setminus \Omega\,,
\end{array} \right.
$$
where $s\in (0,1)$ is fixed, $(-\Delta)^s$ is the fractional Laplace operator, $\lambda$ is a real parameter, $\Omega\subset \RR^n$,  $n>2s$, is an open bounded set with continuous boundary and nonlinearity $f$ satisfies natural superlinear and subcritical growth assumptions. Precisely, along the paper we prove the existence of at least three non-trivial solutions for this problem in a suitable left neighborhood of any eigenvalue of $(-\Delta)^s$. At this purpose we employ a variational theorem of mixed type (one of the so-called $\nabla$-theorems).
\end{abstract}

\maketitle

\tableofcontents

\section{Introduction}\label{sec:introduzione}

Classical critical point theorems, like the Mountain Pass
Theorem, the Linking Theorem or the Saddle Point Theorem (see \cite{ar, puccirad, rabinowitz, struwe}),
have been extensively used in order to construct non-trivial solutions for nonlocal equations of the type
$$ \left\{
\begin{array}{ll}
(-\Delta)^s u=f(x,u) & {\mbox{ in }} \Omega\\
u=0 & {\mbox{ in }} \RR^n\setminus \Omega\,
\end{array} \right.
$$
under different growth assumptions on $f$  (see, e.g., \cite{colorado, barrioscoloradoservadei,capella, fiscella, fsvBS, molicaservadei, sY, sBNRES, svmountain, svlinking, servadeivaldinociBN, servadeivaldinociBNLOW, servadeivaldinociCFP, tan} and references therein, and \cite{cabretan} for a minimization procedure).
Here, $s\in (0,1)$ is fixed
and $(-\Delta)^s$ is the fractional Laplace operator,
which (up to normalization factors) may be defined as
\begin{equation} \label{2}
-(-\Delta)^s u(x)=
\int_{\RR^n}\frac{u(x+y)+u(x-y)-2u(x)}{|y|^{n+2s}}\,dy\,,
\,\,\,\,\, x\in \RR^n\,.
\end{equation}

The aim of this paper is to focus on the existence of multiple solutions for this kind of problems, in the case when $f$ is a superlinear and subcritical nonlinearity. Precisely, we will study the following problem:
\begin{equation}\label{problemaK}
\left\{
\begin{array}{ll}
\mathcal L_K u+\lambda u+f(x,u)=0 & {\mbox{ in }} \Omega\\
u=0 & {\mbox{ in }} \RR^n\setminus \Omega.
\end{array} \right.
\end{equation}
Here $s\in (0,1)$ is fixed, $n>2s$, $\Omega\subset \RR^n$ is an open bounded set with
continuous boundary, and the non-local operator $\mathcal L_K$ is defined as
\begin{equation}\label{lk}
\mathcal L_Ku(x)=
\int_{\RR^n}\Big(u(x+y)+u(x-y)-2u(x)\Big)K(y)\,dy\,,
\,\,\,\,\, x\in \RR^n,
\end{equation}
and $K:\RR^n\setminus\{0\}\rightarrow(0,+\infty)$ is a function with the properties that
\begin{equation}\label{kernel}
{\mbox{$m K\in L^1(\RR^n)$, where $m(x)=\min \{|x|^2, 1\}$\,;}}
\end{equation}
\begin{equation}\label{kernelfrac}
\mbox{there exists}\,\, \theta>0\,\,
\mbox{such that}\,\, K(x)\geq \theta |x|^{-(n+2s)}\,\,
\mbox{for any}\,\, x\in \RR^n \setminus\{0\}\,.\\
\end{equation}
A model for $K$ is given by $K(x)=|x|^{-(n+2s)}$. In this case $\mathcal L_K$ is the fractional Laplace operator $-(-\Delta)^s$\,.

In the recent papers \cite{fiscellamolicaservadeiCFS, 152, molicaMRL, servadeiKavian} the multiplicity of solutions in the nonlocal fractional setting has been addressed by means of classical critical point theorems in the spirit of the ones cited above.

In \cite{marsac:sns} (see also \cite{MMu, marsac}) Marino and
Saccon introduced new critical point theorems, the so-called \emph{theorems
of mixed type}, or \emph{$\nabla$-theorems}, which allow to get multiplicity results for semilinear elliptic problems (see, for instance, \cite{mms, mug4, mugnaiNODEA, Mug, mugpag, OL, fwang, fwang2, wzz}).

We think that a natural question is whether or not these techniques may be adapted
to the fractional Laplacian setting. One can define a fractional power of the Laplacian using its spectral decomposition: indeed, in \cite{mugpag} the same problem considered along this paper, but for this spectral fractional Laplacian, has been considered. As in \cite{mugpag}, the purpose of this paper is to investigate the existence of multiple weak solutions for~\eqref{problemaK}.

A weak solution of \eqref{problemaK} is a solution of the following problem:
\begin{equation}\label{problemaKweak}
\left\{\begin{array}{l}
{\displaystyle \int_{\RR^n \times \RR^n } (u(x)-u(y))(\varphi(x)-\varphi(y))K(x-y) dx\,dy-\lambda \int_\Omega u(x)\varphi(x)\,dx}\\
\qquad \qquad \qquad \qquad \qquad \qquad \qquad \qquad
{\displaystyle = \int_\Omega f(x,u(x))\varphi(x)dx}\,\,\,\,\,\,\, \forall\,\, \varphi \in X_0\\
\\
u\in X_0
\end{array}\right.
\end{equation}
(for this see \cite[Lemma~5.6]{sv} and \cite[footnote~3]{servadeivaldinociBN}). Here $X_0$ is defined as follows: $X$ is the linear space of Lebesgue
measurable functions from $\RR^n$ to $\RR$ such that the restriction
to $\Omega$ of any function $g$ in $X$ belongs to $L^2(\Omega)$ and
$$\mbox{the map}\,\,\,
(x,y)\mapsto (g(x)-g(y))\sqrt{K(x-y)}\,\,\, \mbox{is in}\,\,\,
L^2\big(\RR^n \times \RR^n \setminus ({\mathcal C}\Omega\times
{\mathcal C}\Omega), dxdy\big)\,,$$
where ${\mathcal C}\Omega:=\RR^n \setminus\Omega$.
Moreover, $$X_0=\{g\in X : g=0\,\, \mbox{a.e. in}\,\,
\RR^n\setminus
\Omega\}\,.$$

As we said here above, we suppose that equation~\eqref{problemaK} is superlinear and subcritical, that is its right-hand side $f:\Omega\times \RR\to \RR$ verifies the following conditions:
\begin{equation}\label{caratheodory}
f\,\, \mbox{is a Carath\'eodory function};
\end{equation}
\begin{equation}\label{crescita}
\begin{aligned}
&\mbox{there exist}\,\, a_1, a_2>0\,\,\mbox{and}\,\, q\in (2, 2^*), 2^*=2n/(n-2s)\,,\,\, \mbox{such that}\\
&\qquad \qquad |f(x,t)|\le a_1+a_2|t|^{q-1}\,\, \mbox{a.e.}\,\, x\in \Omega, t\in \RR\,;
\end{aligned}
\end{equation}
\begin{equation}\label{a3a4}
\begin{aligned}
\mbox{there exist}\,\, & \mbox{two positive constants}\,\, a_3\,\,
\mbox{and}\,\, a_4\,\, \mbox{such that} \\
& F(x,t)\geq a_3|t|^q-a_4\,\, \mbox{a.e.}\,\, x\in \Omega, t\in \RR;
\end{aligned}
\end{equation}
\begin{equation}\label{cond0}
{\displaystyle \lim_{|t|\to 0}\frac{f(x,t)}{|t|}=0}\,\, \mbox{uniformly in}\,\, x\in \Omega\,;
\end{equation}
\begin{equation}\label{mu0}
0<q F(x,t)\le tf(x,t)\,\,\mbox{a.e.}\,\, x\in \Omega, t\in \RR\setminus\{0\},
\end{equation}
where $q$ is given in \eqref{crescita} and the function $F$ is
the primitive of $f$ with respect to the second variable, that is
\begin{equation}\label{F}
{\displaystyle F(x,t):=\int_0^t f(x,\tau)d\tau}\,.
\end{equation}

As a model for $f$ we can take the function $f(x,t)=a(x)|t|^{q-2}t$, with $a\in L^\infty(\Omega)$, $\inf_\Omega a>0$ and $q\in (2, 2^*)$. The exponent $2^*$ here plays the role of a
fractional critical Sobolev exponent (see, e.g. \cite[Theorem~6.5]{valpal}).

\begin{remark}
As remarked in \cite{dmnodea2}, condition \eqref{a3a4} is not a mere consequence of \eqref{mu0}, and must be assumed {\em a priori}, unless $f:\bar \Omega \times \RR\to \RR$ is continuous and \eqref{crescita} holds for every $(x,t)\in \bar \Omega \times \RR$.
\end{remark}

The main result of the present paper can be stated as follows:

\begin{theorem}\label{thmain}
Let $s\in (0,1)$, $n>2s$ and $\Omega$ be an open bounded subset of
$\RR^n$ with continuous boundary. Let
$K:\RR^n\setminus\{0\}\rightarrow(0,+\infty)$ be a function
satisfying \eqref{kernel} and \eqref{kernelfrac} and let $f:\Omega\times \RR \to \RR$ be a
function verifying \eqref{caratheodory}--\eqref{mu0}.

Then, for every eigenvalue $\lambda_k$ of $-\mathcal L_K$ with homogeneous Dirichlet boundary data, there exists a left neighborhood $\mathcal O_k$ of $\lambda_k$ such that problem~\eqref{problemaK} admits at least three non-trivial weak solutions for all $\lambda\in \mathcal O_k$.
\end{theorem}

In the non-local framework, the simplest example we can deal with is
given by the fractional Laplacian, according to the following
result:
\begin{theorem}\label{thlapfrac}
Let $s\in (0,1)$, $n>2s$ and $\Omega$ be an open bounded subset of $\RR^n$
with continuous boundary.
If $f:\Omega\times \RR \to \RR$ is a function verifying \eqref{caratheodory}--\eqref{mu0}, then the problem
\begin{equation}\label{problema}
\left\{
\begin{array}{ll}
(-\Delta)^s u-\lambda u=f(x,u) & {\mbox{ in }} \Omega\\
u=0 & {\mbox{ in }} \RR^n\setminus \Omega
\end{array} \right.
\end{equation}
admits at least three non-trivial weak solutions belonging to $H^s(\RR^n)$ in a suitable left neighborhood of any eigenvalue of $(-\Delta)^s$ with homogeneous Dirichlet boundary data.
\end{theorem}

When $s=1$, problem~\eqref{problema} reduces to a standard semilinear Laplace equation:
in this sense Theorem~\ref{thlapfrac} may be seen as the fractional version
of the result in \cite[Theorem~1]{mugnaiNODEA}.

We prove Theorem~\ref{thmain} employing variational and topological methods. Precisely, we apply a $\nabla$-theorem due to Marino and Saccon in \cite{marsac:sns}, see Theorem~\ref{thmarinosaccon} below. The main difficulties in applying such a theorem are obviously related to the nonlocal nature of the problem.

The paper is organized as follows. In Section~\ref{sec:preliminary} we collect the notation and some
preliminary observations. In Section~\ref{sec:compactness} we discuss the compactness property of the Euler-Lagrange functional associated with the problem under consideration, while  Section~\ref{sec:geometry} is devoted to its geometric structure. In Section~\ref{sec:nablacondition}
we prove the $\nabla$-condition, which is one of the main ingredient of the critical point theorem we employ in order to get our multiplicity result.
Finally, in Section~\ref{sec:proofthmain} we prove Theorem~\ref{thmain}.

\section{Preliminaries}\label{sec:preliminary}
In this section we give some preliminary results.

\subsection{Variational setting}\label{subsec:variational}
First of all, we need some notations. In the sequel we endow the space~$X_0$ with the norm defined as (see \cite[Lemma~6]{svmountain})
\begin{equation}\label{norma}
\|g\|_{X_0}=\Big(\int_{\RR^n\times \RR^n} |g(x)-g(y)|^2K(x-y)dx\,dy\Big)^{1/2},
\end{equation}
which is obviously related to the so-called \emph{Gagliardo norm}
\begin{equation}\label{gagliardonorm}
\|g\|_{H^s(\Omega)}=\|g\|_{L^2(\Omega)}+
\Big(\int_{\Omega\times
\Omega}\frac{\,\,\,|g(x)-g(y)|^2}{|x-y|^{n+2s}}\,dx\,dy\Big)^{1/2}
\end{equation}
of the usual fractional
Sobolev space $H^s(\Omega)$.
For further details on the fractional Sobolev spaces we refer to~\cite{adams, valpal, molicalibro}
and to the references therein.

Problem~\eqref{problemaKweak} has a variational structure: indeed, it is the Euler-Lagrange equation of the functional~$\mathcal J_\lambda:X_0\to \RR$ defined as
$$\mathcal J_\lambda(u)=\frac 1 2 \int_{\RR^n\times \RR^n}|u(x)-u(y)|^2 K(x-y)\,dx\,dy-\frac \lambda 2 \int_\Omega |u(x)|^2\,dx-\int_\Omega F(x, u(x))dx\,.$$
Note that when functional~$\mathcal J_\lambda$ is Fr\'echet differentiable at $u\in X_0$, we have that for any $\varphi\in X_0$
$$\begin{aligned}
\langle \mathcal J'_\lambda(u), \varphi\rangle & = \int_{\RR^n\times \RR^n} \big(u(x)-u(y)\big)\big(\varphi(x)-\varphi(y)\big)K(x-y)\,dx\,dy-\lambda \int_\Omega u(x)\varphi(x)\,dx\\
& \qquad \qquad \qquad \qquad \qquad \qquad\qquad \qquad \qquad \qquad  -\int_\Omega f(x, u(x))\varphi(x)\,dx\,,
\end{aligned}$$
where we have denoted by $\langle \cdot,\cdot\rangle$ the duality between $X_0'$ and $X_0$.
Thus, critical points of $\mathcal J_\lambda$ are solutions to problem~\eqref{problemaKweak}.
In order to find multiplicity of critical points, we will use the $\nabla$-theorem in the form of Theorem~\ref{thmarinosaccon} (see Section~\ref{sec:proofthmain}).

\subsection{Estimates on the nonlinearity} Here, we recall some estimates on the nonlinear term and its primitive, which will be useful in the sequel.
These estimates are quite standard and do not take into account the non-local
features of the problem. For a proof we refer to \cite[Lemma~3 and Lemma~4]{svmountain}.

By assumptions~\eqref{caratheodory}, \eqref{crescita} and \eqref{cond0} we deduce that
\begin{equation}\label{cond22}
\begin{aligned}
& \mbox{for any}\,\, \varepsilon>0\,\, \mbox{there exists}\,\, C_{\varepsilon}>0\,\, \mbox{such that}\\
& \quad |f(x,t)|\leq 2\varepsilon |t|+qC_{\varepsilon} |t|^{q-1}\,\, \mbox{a.e.}\,\, x\in \Omega, t\in \RR
\end{aligned}
\end{equation}
and so, as a consequence,
\begin{equation}\label{cond22F}
|F(x,t)|\leq \varepsilon\,|t|^2+C_{\varepsilon}\, |t|^{q}\,,
\end{equation}
where $F$ is defined as in \eqref{F}\,. This implies that functional~$\mathcal J_\lambda$ is Fr\'echet differentiable at any point $u\in X_0$.

\subsection{An eigenvalue problem for $-\mathcal L_K$}\label{subsec:eigenvalue}
Along the present paper we consider the following eigenvalue problem associated to the integro-differential operator~$-\mathcal L_K$:
\begin{equation}\label{problemaautovalori}
\left\{\begin{array}{ll}
-\mathcal L_K u=\lambda\, u & \mbox{in } \Omega\\
u=0 & \mbox{in } \RR^n\setminus \Omega\,.
\end{array}\right.
\end{equation}
We denote by $\big\{ \lambda_k\big\}_{{k\in\NN}}$ the sequence of the eigenvalues of the problem~\eqref{problemaautovalori}, with
\begin{equation}\label{lambdacrescente}
0<\lambda_1<\lambda_2\le \dots \le \lambda_k\le \lambda_{k+1}\le \dots
\end{equation}
$${\mbox{$\lambda_k\to +\infty$ as $ k\to +\infty\,,$}}$$
and by $e_k$ the eigenfunction corresponding to $\lambda_k$. Moreover, we normalize $\big\{e_k\big\}_{{k\in\NN}}$ in such a way that
this sequence provides an orthonormal basis of $L^2(\Omega)$ and an orthogonal basis of $X_0$\,. For a complete study of the spectrum of the integro-differential operator~$-\mathcal L_K$ we refer to \cite[Proposition~2.3]{sY}, \cite[Proposition~9 and Appendix~A]{svlinking} and \cite[Proposition~4]{servadeivaldinociBNLOW}\,. In particular, we recall that all eigenfunctions are H\"older continuous up to the boundary of $\Omega$.

Finally, we say that eigenvalue $\lambda_k$, $k\geq 2$, has multiplicity $m\in \NN$ if
$$\lambda_{k-1}<\lambda_k=\dots =\lambda_{k+m-1}<\lambda_{k+m}\,.$$
In this case the set of all the eigenvalues corresponding to $\lambda_k$ agrees with
$$\mbox{span}\left\{e_k,\ldots,e_{k+m-1}\right\}\,.$$

In the following, for any $k\in \NN$, we set
$$\mathbb H_k=\mbox{span}\left\{e_{1},\ldots,e_k\right\}$$
and
$$\mathbb P_k=
\left\{u\in X_{0}:\,\,\left\langle u,e_{j}\right\rangle_{X_{0}}=0\,\,\,\,\mbox{for any}\,\, j=1,\ldots,k \right\},$$
where
\[
\left\langle u,v\right\rangle_{X_{0}}:= \int_{\RR^n\times \RR^n} \big(u(x)-u(y)\big)\big(\varphi(x)-\varphi(y)\big)K(x-y)\,dx\,dy
\]
makes $X_0$ a Hilbert space, see \cite[Lemma~7]{svmountain}. In this way, the variational characterization of the eigenvalues (see \cite[Proposition~9]{svlinking}) implies that
\begin{equation}\label{poincarevincolata}
\int_{\RR^n\times \RR^n}|u(x)-u(y)|^2 K(x-y)\,dx\,dy\geq \lambda_{k+1} \int_\Omega |u(x)|^2\,dx \ \mbox{ for all $u\in \mathbb P_k$},
\end{equation}
while, using the orthogonality properties of the eigenvalues, a standard Fourier decomposition gives
\begin{equation}\label{antipoincare}
\int_{\RR^n\times \RR^n}|u(x)-u(y)|^2 K(x-y)\,dx\,dy\leq \lambda_k \int_\Omega |u(x)|^2\,dx \ \mbox{ for all $u\in \mathbb H_k$}.
\end{equation}

\subsection{Gradient in $X_0$}\label{secgrad}

Being $X_0$ a Hilbert space and $\mathcal J_\lambda$ of class $C^1$, the gradient $\nabla \mathcal J_\lambda$ of $\mathcal J_\lambda$ is immediately defined as
\begin{equation}\label{gradJ}
\begin{aligned}
\langle \nabla \mathcal J_\lambda(u),v\rangle_{X_0} :&=\langle \mathcal J'_\lambda(u), v\rangle\\
& =\langle u,v\rangle_{X_0}-\lambda \int_\Omega u(x)v(x)\,dx-\int_\Omega f(x,u(x))v(x)\,dx
\end{aligned}
\end{equation}
for any $u,v\in X_0$.

Let $\nu \in [1, 2^*]$ and $\nu'$ be its conjugate, that is $1/\nu+1/\nu'=1$.
Introducing the operator $\mathcal L_K^{-1}:L^{\nu'}(\Omega)\to X_0$, defined as
$\mathcal L_K^{-1}g=v$ if and only if $v\in X_0$ solves
\[
\begin{cases}
\mathcal L_Kv=g(x) & \mbox{ in } \Omega,\\
v=0 & \mbox{ in } \RR^n\setminus \Omega,
\end{cases}
\]
it is readily seen that
\begin{equation}\label{grad}
\nabla \mathcal J_\lambda(u)=u-\mathcal L_K^{-1}(\lambda u+f(x,u))
\end{equation}
for all $u\in X_0$. Indeed, setting $w=\mathcal L_K^{-1}(\lambda u+f(x,u))$ and using the definitions of $\mathcal L_K$ and $\mathcal J_\lambda$, for any test function $\varphi\in X_0$ we get that
$$\begin{aligned}
\langle w, \varphi \rangle_{X_0} & = \lambda \int_\Omega u(x)\varphi(x)\,dx +\int_\Omega f(x,u(x))\varphi(x)\,dx\\
& = - \langle \mathcal J_\lambda'(u), \varphi\rangle+\langle u, \varphi \rangle_{X_0}\\
& = -\langle \nabla \mathcal J_\lambda(u),v\rangle_{X_0}+\langle u, \varphi \rangle_{X_0}\,,
\end{aligned}$$
that is
$$w=-\nabla \mathcal J_\lambda(u)+u\,,$$
which gives the assertion.

Moreover, we note that $\mathcal L_K^{-1}$ is a compact operator for all $\nu\in [1,2^*)$. Indeed, if $\{g_n\}_{n\in \NN}$ is bounded in $L^{\nu'}(\Omega)$, standard calculations imply that $\left\{\mathcal L_K^{-1}g_n\right\}_{n\in \NN}$ is relatively compact in~$X_0$.

For further use, we also note that, again using the definition of $\mathcal L_K^{-1}$,
\begin{equation}\label{scarica}
\langle u, \mathcal L_K^{-1}v\rangle_{X_0}=\int_\Omega u(x)v(x)\,dx
\end{equation}
for every $u,v\in X_0$.

\section{Compactness condition}\label{sec:compactness}
In this section we check the validity of the \emph{Palais--Smale condition} for functional~$\mathcal J_\lambda$ at any level, that is we prove that for each $c\in \RR$ any Palais--Smale sequence for $\mathcal J_\lambda$ at level~$c$ admits a subsequence which is strongly convergent in $X_0$. As usual, we say that $\{u_j\}_{j\in \NN}\subset X_0$ is a \emph{Palais--Smale sequence} for $\mathcal J_\lambda$ at level~$c\in \RR$ if
\begin{equation}\label{Jc0}
\mathcal J_\lambda(u_j)\to c
\end{equation}
and
\begin{equation}\label{J'00}
\sup\Big\{ \big|\langle\,\mathcal J'_\lambda(u_j),\varphi\,\rangle \big|\,: \;
\varphi\in
X_0\,, \|\varphi\|_{X_0}=1\Big\}\to 0
\end{equation}
as $j\to +\infty$.

\begin{proposition}\label{lemmaPS}
Let $\lambda>0$ and $f$ be a function satisfying conditions~\eqref{caratheodory}--\eqref{mu0}.

Then, functional~$\mathcal J_\lambda$ satisfies the Palais--Smale condition at any level~$c\in \RR$\,.
\end{proposition}
\begin{proof}
Let $c\in \RR$ and let $\{u_j\}_{j\in \NN}$ be a Palais--Smale
sequence for $\mathcal J_\lambda$ at level $c$.
First of all, let us prove that
\begin{equation}\label{ujlimitata}
\mbox{the sequence}\,\,\, \{u_j\}_{j\in \NN} \,\,\, \mbox{is bounded in} \,\,\, X_0\,.
\end{equation}
For this purpose, we note that by \eqref{Jc0} and \eqref{J'00} for any $j\in \NN$  it easily follows
that there exists $\kappa>0$ such that
$$\Big|\langle \mathcal J'_\lambda(u_j), \frac{u_j}{\|u_j\|_{X_0}}\rangle\Big| \leq \kappa\,,$$
and
$$|\mathcal J_\lambda(u_j)|\leq \kappa\,,$$
so that, as a consequence of these two relations, we also have
\begin{equation}\label{kappaPS}
\mathcal J_\lambda(u_j)-\frac 1 \mu \langle \mathcal J'_\lambda(u_j), u_j\rangle\leq \kappa \left(1+ \|u_j\|_{X_0}\right)\,,
\end{equation}
where $\mu$ is a parameter such that $\mu\in (2, q)$\,.

Now, thanks to \eqref{mu0} and \eqref{a3a4} we get
\begin{equation}\label{jj'0L}
\begin{aligned}
\mathcal J_\lambda(u_j)-\frac 1 \mu \langle \mathcal J_\lambda'(u_j), u_j\rangle & = \left(\frac 1 2 -\frac 1 \mu\right)\Big(\|u_j\|_{X_0}^2-\lambda \|u_j\|_{L^2(\Omega)}^2\Big)\\
& \qquad  -\int_\Omega \Big(F(x, u_j(x))-\frac 1 \mu\, f(x, u_j(x)) \,u_j(x)\Big)\,dx\\
& \geq \left(\frac 1 2 -\frac 1 \mu\right)\Big(\|u_j\|_{X_0}^2-\lambda \|u_j\|_{L^2(\Omega)}^2\Big)\\
& \qquad +\left(\frac q \mu-1\right)\int_\Omega F(x, u_j(x))\,dx\\
& \geq \left(\frac 1 2 -\frac 1 \mu\right)\Big(\|u_j\|_{X_0}^2-\lambda \|u_j\|_{L^2(\Omega)}^2\Big)\\
& \qquad +a_3\left(\frac q \mu-1\right)\|u_j\|_{L^q(\Omega)}^q-a_4\left(\frac q \mu-1\right)\,|\Omega|\,.
\end{aligned}
\end{equation}
Note that for any $\varepsilon>0$ the Young inequality gives
\begin{equation}\label{youngPS}
\|u_j\|_{L^2(\Omega)}^2\leq \frac{2\varepsilon}{q}\,\|u_j\|_{L^q(\Omega)}^q+\frac{q-2}{q}\,\varepsilon^{-2/(q-2)}\,|\Omega|\,,
\end{equation}
so that, by \eqref{jj'0L} and \eqref{youngPS}, we obtain that
\begin{equation}\label{young2}
\begin{aligned}
\mathcal J_\lambda(u_j)-\frac 1 \mu \langle \mathcal J_\lambda'(u_j), u_j\rangle &
\geq \left(\frac 1 2 -\frac 1 \mu\right)\|u_j\|_{X_0}^2-\lambda\left(\frac 1 2 -\frac 1 \mu\right) \frac{2\varepsilon}{q}\,\|u_j\|_{L^q(\Omega)}^q\\
& \qquad -\lambda\left(\frac 1 2 -\frac 1 \mu\right) \frac{q-2}{q}\,\varepsilon^{-2/(q-2)}\,|\Omega|\\
& \qquad +a_3\left(\frac q \mu-1\right)\|u_j\|_{L^q(\Omega)}^q-a_4\left(\frac q \mu-1\right)\,|\Omega|\\
& = \left(\frac 1 2 -\frac 1 \mu\right)\|u_j\|_{X_0}^2\\
& \qquad +\Big[a_3\left(\frac q \mu-1\right)-\lambda\left(\frac 1 2 -\frac 1 \mu\right)\frac{2\varepsilon}{q}\Big]\|u_j\|_{L^q(\Omega)}^q-C_\varepsilon\,,
\end{aligned}
\end{equation}
where $C_\varepsilon$ is a constant such that $C_\varepsilon\to +\infty$ as $\varepsilon\to 0$, being $q>\mu>2$\,.

Now, choosing $\varepsilon$ so small that
$$a_3\left(\frac q \mu-1\right)-\lambda\left(\frac 1 2 -\frac 1 \mu\right)\frac{2\varepsilon}{q}>0\,,$$
by \eqref{young2} we get
\begin{equation}\label{jj'20}
\mathcal J_\lambda(u_j)-\frac 1 \mu \langle \mathcal J_\lambda'(u_j), u_j\rangle \geq \left(\frac 1 2 -\frac 1 \mu\right)\|u_j\|_{X_0}^2-C_\varepsilon\,.
\end{equation}

Finally, by \eqref{kappaPS} and \eqref{jj'20} for any $j\in \NN$
$$\|u_j\|_{X_0}^2 \leq \kappa_*\left(1+\|u_j\|_{X_0}\right)$$
for a suitable positive constant $\kappa_*$\,. Hence, assertion~\eqref{ujlimitata} is proved.

Now, let us finish the proof of the Palais--Smale condition for $\mathcal J_\lambda$\,.
Since $\{u_j\}_{j\in \NN}$ is bounded in $X_0$ and $X_0$ is a reflexive space, up to a
subsequence, still denoted by $u_j$, there exists $u_\infty \in X_0$
such that
\begin{equation}\label{convergenze0}
\begin{aligned}
 & \int_{\RR^n\times \RR^n}\big(u_j(x)-u_j(y)\big)\big(\varphi(x)-\varphi(y)\big) K(x-y)\,dx\,dy \to \\
& \qquad \qquad
\int_{\RR^n\times \RR^n}\big(u_\infty(x)-u_\infty(y)\big)\big(\varphi(x)-\varphi(y)\big)
K(x-y)\,dx\,dy  \quad \mbox{for any}\,\, \varphi\in X_0\,,
\end{aligned}
\end{equation}
while, by \cite[Lemma~8]{svmountain}, up to a subsequence,
\begin{equation}\label{convergenze0bis}
\begin{aligned}
& u_j \to u_\infty \quad \mbox{in}\,\, L^\nu(\RR^n)\quad \mbox{for any}\,\,\nu\in [1, 2^*) \\
& u_j \to u_\infty \quad \mbox{a.e. in}\,\, \RR^n
\end{aligned}
\end{equation}
as $j\to +\infty$\,. Finally, by \cite[Theorem~IV.9]{brezis} we know that for any $\nu\in [1, 2^*)$ there exists $\ell_\nu\in L^\nu(\RR^n)$ such that
\begin{equation}\label{dominata20}
|u_j(x)|\leq \ell_\nu(x) \quad \mbox{a.e. in}\,\, \RR^n\,\quad \mbox{for any}\,\,j\in \NN\,.
\end{equation}

By \eqref{crescita}, \eqref{convergenze0bis}, \eqref{dominata20}, the
fact that the map $t\mapsto f(\cdot, t)$ is continuous in $t\in \RR$ (see \eqref{caratheodory})
and the Dominated Convergence Theorem, we get
\begin{equation}\label{convf0}
\int_\Omega f(x, u_j(x))u_j(x)\,dx \to \int_\Omega f(x, u_\infty(x))u_\infty(x)\,dx
\end{equation}
and
\begin{equation}\label{convfu0}
\int_\Omega f(x, u_j(x))u_\infty(x)\,dx \to \int_\Omega f(x, u_\infty(x))u_\infty(x)\,dx
\end{equation}
as $j\to +\infty$. Furthermore, by \eqref{J'00} and \eqref{ujlimitata} we have that
$$\begin{aligned}
 0\leftarrow \langle \mathcal J'_\lambda(u_j), u_j\rangle & = \int_{\RR^n\times \RR^n}|u_j(x)-u_j(y)|^2 K(x-y)\,dx\,dy -\lambda \int_\Omega |u_j(x)|^2\,dx\\
& \qquad \qquad \qquad \qquad \qquad \qquad \qquad - \int_\Omega f(x, u_j(x))u_j(x)\,dx
\end{aligned}$$
so that, by \eqref{convergenze0bis} and \eqref{convf0} we deduce that
\begin{equation}\label{convnorma10}
\int_{\RR^n\times \RR^n}|u_j(x)-u_j(y)|^2 K(x-y)\,dx\,dy\to \lambda \int_\Omega |u_\infty(x)|^2\,dx+\int_\Omega f(x, u_\infty(x))u_\infty(x)\,dx
\end{equation}
as $j\to +\infty$\,, while, by \eqref{J'00}, \eqref{convergenze0} (both with test function $\varphi=u_\infty$), \eqref{convergenze0bis} and \eqref{convfu0}, we get
\begin{equation}\label{convnorma20}
\begin{aligned}
\int_{\RR^n\times \RR^n}|u_\infty(x)-u_\infty(y)|^2 K(x-y)\,dx\,dy & = \lambda \int_\Omega |u_\infty(x)|^2\,dx\\
& \qquad \qquad +\int_\Omega f(x, u_\infty(x))u_\infty(x)\,dx\,.
\end{aligned}
\end{equation}
Hence, \eqref{convnorma10} and \eqref{convnorma20} give that
\begin{equation}\label{convnormaX00}
\|u_j\|_{X_0}\to \|u_\infty\|_{X_0}
\end{equation}
as $j\to +\infty$. With this, it is easy to see that
$$\begin{aligned}
\|u_j-u_\infty\|_{X_0}^2 & = \|u_j\|_{X_0}^2 + \|u_\infty\|_{X_0}^2 -2 \int_{\RR^n\times \RR^n} \big(u_j(x)-u_j(y)\big)\big(u_\infty(x)-u_\infty(y)\big) K(x-y)\,dx\,dy\\
& \to 2\|u_\infty\|_{X_0}^2-2\int_{\RR^n\times \RR^n}|u_\infty(x)-u_\infty(y)|^2 K(x-y)\,dx\,dy=0
\end{aligned}$$
as $j\to +\infty$, thanks to \eqref{convergenze0} and \eqref{convnormaX00}\,.
Then, the proof of Proposition~\ref{lemmaPS} is complete.
\end{proof}

\section{Geometry of the $\nabla$-theorem}\label{sec:geometry}
In this section we check that functional~$\mathcal J_\lambda$ has the geometric
structure required by the $\nabla$-theorem stated in Theorem~\ref{thmarinosaccon} (see Section~\ref{sec:proofthmain}). Precisely, we want to show that if there exist $k$ and $m$ in $\NN$ such that
$$\lambda_{k-1}<\lambda<\lambda_k=\dots = \lambda_{k+m-1}<\lambda_{k+m}$$
and $\lambda$ is sufficiently close to $\lambda_k$, then functional~$\mathcal J_\lambda$ agrees with the geometric framework of Theorem~\ref{thmarinosaccon}, taking
$$\begin{aligned}
& X_1:=\mathbb H_{k-1}\\
& X_2:=\mbox{span}\left\{e_k,\ldots,e_{k+m-1}\right\}\\
& X_3:=\mathbb P_{k+m-1}\,.
\end{aligned}$$

\begin{proposition}\label{propgeometria}
Let $k$ and $m$ in $\NN$ be such that $\lambda_{k-1}<\lambda<\lambda_k=\dots = \lambda_{k+m-1}<\lambda_{k+m}$ and let $f$ be a function satisfying conditions~\eqref{caratheodory}--\eqref{mu0}.

Then, there exist $\rho, R$, with $R>\rho>0$, such that
$${\displaystyle \sup_{\{u\in X_1, \|u\|_{X_0}\leq R\} \cup \{u\in X_1\oplus X_2 : \|u\|=R\}}} \mathcal J_\lambda(u)<{\displaystyle \inf_{\{u\in X_2\oplus X_3 : \|u\|_{X_0}=\rho\}}\mathcal J_\lambda(u)}\,.$$

\end{proposition}
\begin{proof}
First of all, let us show that
\begin{equation}\label{infpositivo}
{\displaystyle \inf_{\{u\in X_2\oplus X_3 : \|u\|_{X_0}=\rho\}}\mathcal J_\lambda(u)}>0\,.
\end{equation}
For this purpose, let $u$ be a function in $X_2 \oplus X_3=\mathbb P_{k-1}$. By \eqref{cond22F}, we get that for any $\varepsilon>0$
\begin{equation}\label{infJfrac0}
\begin{aligned}
\mathcal J_\lambda(u)& \geq \frac 1 2\int_{\RR^n\times \RR^n}|u(x)-u(y)|^2 K(x-y)dx\,dy-\frac \lambda 2\int_\Omega |u(x)|^2\,dx\\
& \qquad \qquad \qquad \qquad -\varepsilon\int_\Omega |u(x)|^2dx - C_{\varepsilon}\int_\Omega |u(x)|^q\,dx\,.
\end{aligned}
\end{equation}
Moreover, by \eqref{poincarevincolata}, we get that for any $u\in \mathbb P_{k-1}$
$$
\lambda_k\,\int_\Omega |u(x)|^2\,dx \leq \int_{\RR^n\times \RR^n}|u(x)-u(y)|^2 K(x-y)dx\,dy\,,
$$
so that this inequality and \eqref{infJfrac0} give
\begin{equation}\label{infJfrac0bis}
\begin{aligned}
\mathcal J_\lambda(u)& \geq \frac 1 2\left(1-\frac{\lambda}{\lambda_k}\right)\|u\|_{X_0}^2-\varepsilon\|u\|_{L^2(\Omega)}^2
 -C_{\varepsilon}\|u\|_{L^q(\Omega)}^q \\
& \geq \frac 1 2\left(1-\frac{\lambda}{\lambda_k}\right)\|u\|_{X_0}^2-\varepsilon|\Omega|^{(2^*-2)/2^*}
\|u\|_{L^{2^*}(\Omega)}^2 -|\Omega|^{(2^*-q)/2^*}C_{\varepsilon}\|u\|_{L^{2^*}(\Omega)}^q\,.
\end{aligned}
\end{equation}
Here we used also the fact that $L^{2^*}(\Omega)\hookrightarrow L^\nu(\Omega)$ continuously, being
$\Omega$ bounded and $\nu\in [2, 2^*)$ (here $\nu$ takes the values $2$ and $q$).

Using \eqref{kernelfrac}, \eqref{infJfrac0bis} and \cite[Lemma~6]{svmountain}, we
obtain that for any $\varepsilon>0$,
\begin{equation}\label{infJfrac20}
\begin{aligned}
\mathcal J_\lambda(u) & \geq \frac 1 2\left(1-\frac{\lambda}{\lambda_k}\right)\|u\|_{X_0}^2-\varepsilon c|\Omega|^{(2^*-2)/2^*}
 \int_{\RR^n \times \RR^n}\frac{|u(x)-u(y)|^2}{|x-y|^{n+2s}}\, dx\,dy \\
& \qquad -C_{\varepsilon}c^{q/2}|\Omega|^{(2^*-q)/2^*} \left( \int_{\RR^n \times \RR^n}\frac{|u(x)-u(y)|^2}{|x-y|^{n+2s}}\, dx\,dy\right)^{q/2}\\
& \ge \left[\frac 1 2\left(1-\frac{\lambda}{\lambda_k}\right) -\frac{\varepsilon
c|\Omega|^{(2^*-2)/2^*}}{\theta}\right]\int_{\RR^n\times \RR^n} |u(x)-u(y)|^2 K(x-y)\,dx\,dy\\
& \qquad -\frac{C_{\varepsilon}c^{q/2}|\Omega|^{(2^*-q)/2^*}}{\theta} \left( \int_{\RR^n\times \RR^n} |u(x)-u(y)|^2 K(x-y)\, dx\,dy\right)^{q/2}\,,
\end{aligned}
\end{equation}
where $c$ is a suitable universal positive constant.

Choosing $\varepsilon>0$ such that $$2\varepsilon c
|\Omega|^{(2^*-2)/2^*}<\theta\left(1-\frac{\lambda}{\lambda_k}\right),$$ inequality~\eqref{infJfrac20} reads as
$$\mathcal J_\lambda(u) \geq \alpha \|u\|_{X_0}^2\left(1-\kappa
\|u\|_{X_0}^{q-2}\right)\,,$$
for suitable positive constants $\alpha$ and $\kappa$\,.

Now, let $\rho>0$ be sufficiently small, i.e. $\rho$ such that $1-\kappa\rho^{q-2}>0$\,. Then, for any $u\in \mathbb P_{k-1}$ such that $\|u\|_{X_0}=\rho$ we get that
$$\mathcal J_\lambda(u)\geq \alpha \rho^2(1-\kappa \rho^{q-2})>0\,,$$
so that \eqref{infpositivo} is proved.

Now, let us show that
\begin{equation}\label{supnegativo}
{\displaystyle \sup_{\{u\in X_1, \|u\|_{X_0}\leq R\} \cup \{u\in X_1\oplus X_2 : \|u\|=R\}}} \mathcal J_\lambda(u)\leq 0\,.
\end{equation}
First of all, let us take $u\in \mathbb H_{k-1}$. Then
$$u(x)=\sum_{i=1}^{k-1}u_i e_i(x)\,,$$ with $u_i\in \RR$, $i=1, \dots,
k-1$
and so, by \eqref{antipoincare} and \eqref{mu0}, we deduce that
\begin{equation}\label{J<01}
\begin{aligned}
\mathcal J_\lambda(u) & \leq \frac{\lambda_{k-1}-\lambda}{2}\int_\Omega |u(x)|^2\,dx \leq 0\,,
\end{aligned}
\end{equation}
since $\lambda_{k-1}< \lambda$.

Finally, let us consider $u\in X_1\oplus X_2=\mathbb H_k$. By \eqref{a3a4} we have
$$
\mathcal J_\lambda(u)  \leq \frac 1 2\|u\|_{X_0}^2-a_3\|u\|_{L^q(\Omega)}^q+a_4|\Omega|,
$$
and the claim follows recalling that $q>2$ and that $\mathbb H_k$ is a finite-dimensional subspace of $X_0$. This and \eqref{J<01} give \eqref{supnegativo}.

Then, the assertion of Proposition~\ref{propgeometria} comes trivially from \eqref{infpositivo} and \eqref{supnegativo}.
\end{proof}

\section{$\nabla$-condition}\label{sec:nablacondition}
One of the main ingredient of the $\nabla$-theorem (see \cite[Theorem~2.10]{marsac:sns}) we employ in order to get our multiplicity result is the so-called $\nabla$-condition introduced in \cite[Definition~2.4]{marsac:sns}. This section is devoted to the verification of this condition for functional~$\mathcal J_\lambda$. For this purpose, in the sequel we denote by
$$P_C:X_0\to C$$
the orthogonal projection of $X_0$ onto $C$.

Let $C$ be a closed subspace of $X_0$ and $a,b\in \RR\cup\{-\infty, +\infty\}$. We say that functional~$\mathcal J_\lambda$ verifies condition~$(\nabla)(\mathcal J_\lambda, C, a, b)$ if there exists $\gamma>0$ such that
$$
\inf\Big\{\|P_C \nabla \mathcal J_\lambda (u)\|_{X_0} : a\leq \mathcal J_\lambda(u)\leq b,\,\, dist (u, C)\leq \gamma\Big\}>0\,.$$

Roughly speaking, the condition~$(\nabla)(\mathcal J_\lambda, C, a, b)$ requires that $\mathcal J_\lambda$ has no critical points $u\in C$ such that $a\leq \mathcal J_\lambda(u)\leq b$, with some uniformity.
In order to prove this condition for $\mathcal J_\lambda$, we need two preliminary lemmas.

\begin{lemma}\label{lemmanabla1}
Let $k$ and $m$ in $\NN$ be such that $\lambda_{k-1}<\lambda_k=\dots = \lambda_{k+m-1}<\lambda_{k+m}$ and let $f$ be a function satisfying conditions~\eqref{caratheodory}--\eqref{mu0}.

Then, for any $\sigma>0$ there exists $\varepsilon_\sigma>0$ such that for any $\lambda\in [\lambda_{k-1}+\sigma, \lambda_{k+m}-\sigma]$ the unique critical point $u$ of $\mathcal J_\lambda$ constrained on $\mathbb H_{k-1}\oplus \mathbb P_{k+m-1}$ and with $\mathcal J_\lambda(u)\in [-\varepsilon_\sigma, \varepsilon_\sigma]$, is the trivial one.
\end{lemma}
\begin{proof}
We argue by contradiction and we suppose that there exists $\bar \sigma>0$, a sequence $\{\mu_j\}_{j\in \NN}$ in $\RR$ with
\begin{equation}\label{lambdaj}
\mu_j\in [\lambda_{k-1}+\bar \sigma, \lambda_{k+m}-\bar \sigma]
\end{equation}
and a sequence $\{u_j\}_{j\in \NN}$ such that
\begin{equation}\label{ujoplus}
u_j\in \mathbb H_{k-1}\oplus \mathbb P_{k+m-1}\setminus\{0\}\,,
\end{equation}
\begin{equation}\label{ujvincolato}
\langle\,\mathcal J'_{\mu_j}(u_j),\varphi\,\rangle =0\qquad \mbox{for any}\,\, \varphi\in \mathbb H_{k-1}\oplus \mathbb P_{k+m-1}\,,
\end{equation}
for any $j\in \NN$, and
\begin{equation}\label{ujvalorecritico}
\begin{aligned}
\mathcal J_{\mu_j}(u_j) & =\frac 1 2 \int_{\RR^n\times \RR^n}|u_j(x)-u_j(y)|^2 K(x-y)\,dx\,dy -\frac{\mu_j}{2} \int_\Omega |u_j(x)|^2\,dx\\
& \qquad \qquad \qquad \qquad \qquad \qquad \qquad \qquad- \int_\Omega F(x, u_j(x))\,dx\to 0
\end{aligned}
\end{equation}
as $j\to +\infty$\,.

%By \eqref{lambdaj} it is trivial to see that $\{\mu_j\}_{j\in \NN}$ is bounded, so that, up to a subsequence (still denoted by $\mu_j$), we can assume that
%\begin{equation}\label{lambdabar}
%\mu_j \to \bar \lambda\quad \mbox{in}\,\, \RR
%\end{equation}
%as $j\to +\infty$, for a suitable $\bar \lambda\in [\lambda_{k-1}+\bar \sigma, \lambda_{k+m}-\bar \sigma]$\,.

Taking $\varphi=u_j$ in \eqref{ujvincolato} (this is possible thanks to \eqref{ujoplus}) and using \eqref{mu0}, we get that for any $j\in \NN$
$$\begin{aligned}
0 & =\int_{\RR^n\times \RR^n}|u_j(x)-u_j(y)|^2 K(x-y)\,dx\,dy -\mu_j \int_\Omega |u_j(x)|^2\,dx- \int_\Omega f(x, u_j(x))u_j(x)\,dx\\
& = 2\mathcal J_{\mu_j}(u_j)+\int_\Omega \Big(2F(x, u_j(x))-f(x, u_j(x))u_j(x)\Big)\,dx\\
& \leq 2\mathcal J_{\mu_j}(u_j)+(2-q)\int_\Omega F(x, u_j(x))\,dx\,.
\end{aligned}$$
Hence, by this inequality, the fact that $q>2$ and again \eqref{mu0}, we deduce that
$$0<(q-2)\int_\Omega F(x, u_j(x))\,dx\leq 2\mathcal J_{\mu_j}(u_j)\to 0$$
thanks to \eqref{ujvalorecritico}\,, so that we get
\begin{equation}\label{Fujlimite}
\int_\Omega F(x, u_j(x))\,dx\to 0
\end{equation}
as $j\to +\infty$\,.

Now, since \eqref{ujoplus} holds true, for any $j\in\NN$ there exist $v_j\in \mathbb H_{k-1}$ and $w_j\in\mathbb P_{k+m-1}$ such that
$$u_j=v_j+w_j\,.$$
Letting $\varphi=v_j-w_j$ in \eqref{ujvincolato} and taking into account the orthogonality properties of $v_j$ and $w_j$, we have that for any $j\in \NN$
\begin{equation}\label{split}
\begin{aligned}
0 & =\langle\mathcal J'_{\mu_j}(u_j), v_j-w_j\rangle\\
& = \int_{\RR^n\times \RR^n}|v_j(x)-v_j(y)|^2 K(x-y)\,dx\,dy- \int_{\RR^n\times \RR^n}|w_j(x)-w_j(y)|^2 K(x-y)\,dx\,dy\\
& \qquad \qquad -\mu_j \int_\Omega |v_j(x)|^2\,dx + \mu_j \int_\Omega |w_j(x)|^2\,dx- \int_\Omega f(x, u_j(x))(v_j(x)-w_j(x))\,dx\,.
\end{aligned}
\end{equation}
By \eqref{antipoincare} and \eqref{poincarevincolata}, \eqref{split} gives that for any $j\in \NN$
\begin{equation}\label{carvar2}
\begin{aligned}
\int_\Omega f(x, u_j(x))(v_j(x)-w_j(x))\,dx & = \int_{\RR^n\times \RR^n}|v_j(x)-v_j(y)|^2 K(x-y)\,dx\,dy\\
& \qquad \qquad - \int_{\RR^n\times \RR^n}|w_j(x)-w_j(y)|^2 K(x-y)\,dx\,dy\\
& \qquad \qquad -\mu_j \int_\Omega |v_j(x)|^2\,dx + \mu_j \int_\Omega |w_j(x)|^2\,dx\\
& \leq \frac{\lambda_{k-1}-\mu_j}{\lambda_{k-1}}\|v_j\|_{X_0}^2+ \frac{\mu_j-\lambda_{k+m}}{\lambda_{k+m}}\|w_j\|_{X_0}^2\\
& \leq -\frac{\bar \sigma}{\lambda_{k-1}}\|v_j\|_{X_0}^2- \frac{\bar \sigma}{\lambda_{k+m}}\|w_j\|_{X_0}^2\\
&\leq -\frac{\bar \sigma}{\lambda_{k+m}}\|u_j\|_{X_0}^2\,,
\end{aligned}
\end{equation}
thanks again to the properties of the projections $v_j$ and $w_j$ of $u_j$, respectively on $\mathbb H_{k-1}$ and on $\mathbb P_{k+m-1}$\,, and to \eqref{lambdaj}.

On the other hand, by the H\"older inequality, \eqref{crescita} and the fact that $X_0\hookrightarrow L^q(\Omega)$ compactly, there exists a suitable positive constant $\tilde\kappa$, independent of $j$, such that
$$\begin{aligned}
\Big|\int_\Omega f(x, u_j(x))(v_j(x)-w_j(x))\,dx\Big| & \leq
\|f(\cdot, u_j(\cdot))\|_{L^{q/(q-1)}(\Omega)}\, \|v_j-w_j\|_{L^q(\Omega)}\\
& \leq  \tilde\kappa \|f(\cdot, u_j(\cdot))\|_{L^{q/(q-1)}(\Omega)}\, \|v_j-w_j\|_{X_0}\\
& = \tilde\kappa \|f(\cdot, u_j(\cdot))\|_{L^{q/(q-1)}(\Omega)}\, \|u_j\|_{X_0}\,,\\
\end{aligned}$$
so that, by this and \eqref{carvar2}, we deduce that
$$\tilde\kappa \|f(\cdot, u_j(\cdot))\|_{L^{q/(q-1)}(\Omega)}\, \|u_j\|_{X_0}\geq \frac{\bar \sigma}{\lambda_{k+m}}\|u_j\|_{X_0}^2.$$
Being $u_j\not \equiv 0$ by assumption (see \eqref{ujoplus}), we get
\begin{equation}\label{f>}
\|f(\cdot, u_j(\cdot))\|_{L^{q/(q-1)}(\Omega)}\geq \frac{\bar \sigma}{\tilde\kappa \lambda_{k+m}}\|u_j\|_{X_0}
\end{equation}
for any $j\in \NN$\,.

With the previous estimates, we are now ready to show that
\begin{equation}\label{ujboundednabla}
\mbox{the sequence}\,\, \{\|u_j\|_{X_0}\}_{j\in\NN}\,\, \mbox{is bounded in}\,\, \RR.
\end{equation}
For this it is enough to use \eqref{crescita} and \eqref{a3a4}, which yield for any $j\in \NN$
\begin{equation}\label{fujnabla}
\begin{aligned}
\int_\Omega |f(x, u_j(x))|^{q/(q-1)}\,dx & \leq \int_\Omega \left(a_1+a_2|u_j(x)|^{q-1}\right)^{q/(q-1)} \\
& \leq \tilde a_1 +\tilde a_2 \int_\Omega |u_j(x)|^q\,dx\\
& \leq \tilde a_3 + \tilde a_4 \int_\Omega F(x,u_j(x))\,dx\,,
\end{aligned}
\end{equation}
for suitable positive constants $\tilde a_i$, $i=1,\ldots, 4$.
By \eqref{Fujlimite}, \eqref{f>} and \eqref{fujnabla} we get assertion~\eqref{ujboundednabla}\,.

In view of \eqref{ujboundednabla} and \eqref{ujoplus}, we can assume that there exists $u_\infty\in \mathbb H_{k-1}\oplus \mathbb P_{k+m-1}$ such that
\begin{equation}\label{convergenze0nabla}
\begin{aligned}
 & \int_{\RR^n\times \RR^n}\big(u_j(x)-u_j(y)\big)\big(\varphi(x)-\varphi(y)\big) K(x-y)\,dx\,dy \to \\
& \qquad \qquad
\int_{\RR^n\times \RR^n}\big(u_\infty(x)-u_\infty(y)\big)\big(\varphi(x)-\varphi(y)\big)
K(x-y)\,dx\,dy  \quad \mbox{for any}\,\, \varphi\in X_0\,,
\end{aligned}
\end{equation}
while, by \cite[Lemma~8]{svmountain} and \cite[Theorem~IV.9]{brezis}, up to a subsequence,
\begin{equation}\label{convergenze0bisnabla}
\begin{aligned}
& u_j \to u_\infty \quad \mbox{in}\,\, L^q(\RR^n)\\
& u_j \to u_\infty \quad \mbox{a.e. in}\,\, \RR^n
\end{aligned}
\end{equation}
as $j\to +\infty$ and there exists $\ell\in L^q(\RR^n)$ such that
\begin{equation}\label{dominata20nabla}
|u_j(x)|\leq \ell(x) \quad \mbox{a.e. in}\,\, \RR^n\,\quad \mbox{for any}\,\,j\in \NN\,.
\end{equation}

Moreover, by \eqref{cond22} and \eqref{f>} we get that for any $\varepsilon>0$ there exists $C_\varepsilon$ such that
\begin{equation}\label{assurdo1}
\begin{aligned}
0<\frac{\bar \sigma}{\tilde \kappa \lambda_{k+m}} & \leq \frac{\|f(\cdot, u_j(\cdot))\|_{L^{q/(q-1)}(\Omega)}}{\|u_j\|_{X_0}}\\
& \leq \frac{{\displaystyle \left(\int_\Omega \left(2\varepsilon |u_j(x)|+qC_{\varepsilon} |u_j(x)|^{q-1}\right)^{q/(q-1)}\,dx\right)^{(q-1)/q}}}{\|u_j\|_{X_0}}\\
& \leq \frac{\left( 2^{1/(q-1)}\left((2\varepsilon)^{q/(q-1)}\|u_j\|_{L^{q/(q-1)}(\Omega)}^{q/(q-1)}+
(qC_{\varepsilon})^{q/(q-1)}\|u_j\|_{L^q(\Omega)}^q\right)\right)^{(q-1)/q}}{\|u_j\|_{X_0}}\\
& \leq \frac{2\varepsilon\|u_j\|_{L^{q/(q-1)}(\Omega)}+
qC_{\varepsilon}\|u_j\|_{L^q(\Omega)}^{q-1}}{\|u_j\|_{X_0}}\\
& \leq C\left(2\varepsilon+
qC_{\varepsilon}\|u_j\|_{X_0}^{q-2}\right)\,,
\end{aligned}
\end{equation}
thanks to the continuous embedding $X_0\hookrightarrow L^\nu(\Omega)$ for any $\nu\in [1, 2^*)$, and for some universal positive constant $C$.

By \eqref{crescita}, \eqref{F}, \eqref{convergenze0bisnabla}, \eqref{dominata20nabla} and the Dominated Convergence Theorem, it is easily seen that
\begin{equation}\label{Fdominata}
\int_\Omega F(x, u_j(x))\,dx \to \int_\Omega F(x, u_\infty(x))\,dx
\end{equation}
and
\begin{equation}\label{fdominata}
\int_\Omega |f(x, u_j(x))|^{q/(q-1)}\,dx \to \int_\Omega |f(x, u_\infty(x))|^{q/(q-1)}\,dx
\end{equation}
as $j\to +\infty$\,.
Relation~\eqref{Fdominata}, combined with \eqref{mu0}, the fact that $F(x,0)=0$ a.e. $x\in \Omega$, and \eqref{Fujlimite}, yields that
\begin{equation}\label{uinftynabla}
u_\infty \equiv 0\,.
\end{equation}

Now, two cases can occur. If
\begin{equation}\label{uinfty0}
u_j \to u_\infty\equiv 0\qquad \mbox{strongly in}\,\, X_0
\end{equation}
as $j\to +\infty$, then, by \eqref{assurdo1} we get that
$$0<\frac{\bar \sigma}{\tilde\kappa \lambda_{k+m}} \leq 2C\varepsilon\,,$$
which gives a contradiction, due to the fact that $\varepsilon$ is arbitrary.
Otherwise, there exists $\eta>0$ such that $\|u_j\|_{X_0}\geq \eta$ for $j$ large enough.
Then, by this, \eqref{f>}, \eqref{fdominata}, \eqref{uinftynabla} and the fact that $f(x,0)=0$ a.e. $x\in \Omega$ (by \eqref{cond0}), we get that
$$\frac{\bar \sigma\eta}{\tilde\kappa \lambda_{k+m}}\leq 0\,,$$
which is a contradiction. This completes the proof of Lemma~\ref{lemmanabla1}.
\end{proof}

The second lemma we need in order to prove the $\nabla$-condition is the following one:
\begin{lemma}\label{lemmanabla2}
Let $k$ and $m$ in $\NN$ be such that $\lambda_{k-1}<\lambda_k=\dots = \lambda_{k+m-1}<\lambda_{k+m}$,
let $\lambda \in \RR$ and $f$ be a function satisfying conditions~\eqref{caratheodory}--\eqref{mu0}. Moreover, let $\{u_j\}_{j\in \NN}$ be a sequence in $X_0$ such that
\begin{equation}\label{condlemma1}
\{\mathcal J_\lambda(u_j)\}_{j\in \NN}\,\,\, \mbox{is bounded in}\,\,\, \RR\,,
\end{equation}
\begin{equation}\label{condlemma2}
P_{\footnotesize{\rm{span}}\left\{e_k,\ldots,\,e_{k+m-1}\right\}} u_j\to 0\,\,\, \mbox{in}\,\,\, X_0
\end{equation}
and
\begin{equation}\label{condlemma3}
P_{\mathbb H_{k-1}\oplus \mathbb P_{k+m-1}}\nabla \mathcal J_\lambda(u_j)\to 0\,\,\, \mbox{in}\,\,\, X_0
\end{equation}
as $j\to +\infty$\,.

Then, $\{u_j\}_{j\in \NN}$ is bounded in $X_0$\,.
\end{lemma}
\begin{proof}
Assume by contradiction that $\{u_j\}_{j\in \NN}$ is unbounded in $X_0$; without loss of generality, we can assume that
\begin{equation}\label{ujtoinfty}
\|u_j\|_{X_0}\to +\infty
\end{equation}
as $j\to +\infty$ and that there exists $u_\infty\in X_0$ such that
\begin{equation}\label{convnorm}
\begin{aligned}
& \frac{u_j}{\|u_j\|_{X_0}}\rightharpoonup u_\infty\,\,\, \mbox{ in }X_0\\
& \frac{u_j}{\|u_j\|_{X_0}}\to u_\infty\,\,\, \mbox{ in }L^\nu(\Omega) \mbox{ for any $\nu\in[1,2^*)$}
\end{aligned}
\end{equation}
as $j\to +\infty$ and for any $\nu\in[1,2^*)$ there exists $\ell_\nu\in L^\nu(\RR^n)$ such that
\begin{equation}\label{dominata20nablaadd}
|u_j(x)|\leq \ell_\nu(x) \quad \mbox{a.e. in}\,\, \RR^n\,\quad \mbox{for any}\,\,j\in \NN\,.
\end{equation}

Now, for simplicity, we set $P_{\footnotesize{\mbox{span}}\left\{e_k,\ldots,\,e_{k+m-1}\right\}}=:P$ and $P_{\mathbb H_{k-1}\oplus \mathbb P_{k+m-1}}=:Q$, and write
$$u_j=Pu_j+Qu_j\,,$$ where $Pu_j\to 0$ as $j\to\infty$ (see \eqref{condlemma2}). Recalling \eqref{gradJ} and \eqref{grad}, we have
\begin{equation}\label{add11}
\begin{aligned}
\langle Q\nabla \mathcal J_\lambda(u_j),u_j\rangle_{X_0}&= \langle \nabla
\mathcal J_\lambda(u_j),u_j\rangle_{X_0} - \langle P\nabla
\mathcal J_\lambda(u_j),u_j\rangle_{X_0}\\
&=\|u_j\|_{X_0}^2-\lambda \int_\Omega |u_j(x)|^2\,dx
-\int_\Omega f(x,u_j(x))u_j(x)\, dx \\
&-\langle P\big(u_j-\mathcal L_K^{-1}(\lambda
 u_j+f(x,u_j))\big), u_j\rangle_{X_0}.
\end{aligned}
\end{equation}
Since $\langle Pu,v\rangle_{X_0}=\langle u,Pv\rangle_{X_0}$ for any $u,v\in X_0$, we have
\begin{equation}\label{add22}
\begin{aligned}
\langle P\big(u_j-\mathcal L_K^{-1}(\lambda
u_j+f(x,u_j))\big), u_j\rangle_{X_0} & =\|Pu_j\|^2_{X_0}-\lambda \langle Pu_j, \mathcal L_K^{-1}u_j\rangle_{X_0}\\
& \qquad -\langle Pu_j, \mathcal L_K^{-1}f(x,u_j)\rangle_{X_0},
\end{aligned}
\end{equation}
while, by \eqref{scarica}, we obtain
\begin{equation}\label{add33}
\begin{aligned}
\lambda \langle Pu_j, \mathcal L_K^{-1}u_j\rangle_{X_0}+\langle Pu_j, \mathcal L_K^{-1}f(x,u_j)\rangle_{X_0} & =\lambda \int_\Omega |Pu_j(x)|^2\,dx\\
& \qquad +\int_\Omega f(x,u_j(x))Pu_j(x)\,dx.
\end{aligned}
\end{equation}
Therefore, by \eqref{add11}-\eqref{add33}, we get
\begin{equation}\label{sec}
\begin{aligned}
\langle Q\nabla \mathcal J_\lambda(u_j),u_j\rangle_{X_0}& =2\mathcal J_\lambda(u_j)+
2\int_\Omega F(x,u_j(x))\, dx -\int_\Omega f(x,u_j(x))u_j(x)\,dx\\
& \qquad -\|Pu_j\|_{X_0}^2 +\lambda \int_\Omega |Pu_j(x)|^2\, dx+\int_\Omega f(x,u_j(x))Pu_j(x)\,dx.
\end{aligned}
\end{equation}

By \eqref{condlemma1}--\eqref{ujtoinfty} and \eqref{sec} we easily get that
\begin{equation}\label{Ffujto0}
\frac{\displaystyle{2\int_\Omega F(x,u_j(x))\, dx -\int_\Omega f(x,u_j(x))u_j(x)\,dx+\int_\Omega f(x,u_j(x))Pu_j(x)\,dx}}{\|u_j\|_{X_0}^q} \to 0
\end{equation}
as $j\to +\infty$.

Now, let us show that
\begin{equation}\label{uinfty0}
u_\infty\equiv 0\,.
\end{equation}
For this purpose, we firstly claim that
\begin{equation}\label{add44}
\frac{\displaystyle \int_\Omega
f(x,u_j(x))Pu_j(x)\,dx}{\|u_j\|_{X_0}^q} \to 0
\end{equation}
as $j\to +\infty$. Indeed, by \eqref{crescita} and \eqref{dominata20nablaadd}, we have that a.e. $x\in \Omega$
$$\begin{aligned}
|f(x,u_j(x))Pu_j(x)| & \leq \|Pu_j\|_\infty\big(a_1+a_2|u_j(x)|^{q-1}\big)\\
&\leq \|Pu_j\|_\infty\big(a_1+a_2|\ell_q(x)|^{q-1}\big)\,,
\end{aligned}$$
while, by \eqref{condlemma2} and the fact that all norms in $\mathbb H_{k+m-1}$ are equivalent,
$$\|Pu_j\|_\infty\to 0$$
as $j\to +\infty$. Note that $Pu_j\in L^\infty(\Omega)$, since all eigenfunctions of $\mathcal L_K$ are bounded (see \cite[Proposition~2.4]{sY}). Hence, \eqref{add44} holds.

By this and \eqref{Ffujto0} we obtain that
$$0\leftarrow \frac{\displaystyle{2\int_\Omega F(x,u_j(x))\, dx -\int_\Omega f(x,u_j(x))u_j(x)\,dx}}{\|u_j\|_{X_0}^q} \leq \frac{\displaystyle{(2-q)\int_\Omega F(x,u_j(x))\, dx}}{\|u_j\|_{X_0}^q}\leq 0\,,$$
as $j\to +\infty$, also thanks to \eqref{mu0}.
Hence,
\[
\frac{\displaystyle \int_\Omega F(x,u_j(x))\,
dx}{\|u_j\|_{X_0}^q}\to 0
\]
as $j\to +\infty$.

As a consequence of this, \eqref{a3a4} and \eqref{ujtoinfty} we have that
\[
\frac{\displaystyle \int_\Omega |u_j(x)|^q\,
dx}{\|u_j\|_{X_0}^q}\to 0,
\]
as $j\to +\infty$, which yields \eqref{uinfty0}, thanks to \eqref{convnorm}.

Now, by \eqref{condlemma1} and \eqref{ujtoinfty}, we get
$$
\frac{\mathcal J_\lambda(u_j)}{\|u_j\|_{X_0}^2}=\frac 1 2-\frac \lambda 2\frac{\displaystyle \int_\Omega |u_j(x)|^2\,
dx}{\|u_j\|_{X_0}^2}-\frac{\displaystyle \int_\Omega F(x,u_j(x))\,
dx}{\|u_j\|_{X_0}^2}\to 0,
$$
which, together with \eqref{convnorm} (here with $\nu=2$) and \eqref{uinfty0}, implies that
\begin{equation}\label{F2}
\frac{\displaystyle \int_\Omega F(x,u_j(x))\,
dx}{\|u_j\|_{X_0}^2}\to \frac 1 2
\end{equation}
as $j\to +\infty$.

Hence, as a consequence of \eqref{a3a4}, \eqref{ujtoinfty} and \eqref{F2}, there exists $C>0$ such that
\begin{equation}\label{q2}
\|u_j\|_{L^q(\Omega)}^q\leq C\|u_j\|_{X_0}^2 \mbox{ for every $j\in \NN$}.
\end{equation}

Now, let us show that
\begin{equation}\label{fPnorma2}
\frac{\displaystyle \int_\Omega f(x,u_j(x))Pu_j(x)\,dx}{\|u_j\|_{X_0}^2}\to 0
\end{equation}
as $j\to +\infty$.
Indeed, by \eqref{crescita} and the H\"older inequality, we have
$$\begin{aligned}
\frac{\displaystyle \int_\Omega |f(x,u_j(x))Pu_j(x)|\,
dx}{\|u_j\|_{X_0}^2} & \leq
\frac{\|Pu_j\|_\infty}{\|u_j\|_{X_0}^2}\left(a_1|\Omega|+a_2\int_\Omega |u_j(x)|^{q-1}\,
dx\right)\\
& \leq \|Pu_j\|_\infty
\left[\frac{a_1|\Omega|}{\|u_j\|_{X_0}^2}+\frac{a_2'}{\|u_j\|_{X_0}^{2/q}}
\left(\frac{\displaystyle\int_\Omega |u_j(x)|^q
dx}{\|u_j\|_{X_0}^2}\right)^{1-1/q}\right]\\
& \leq \|Pu_j\|_\infty
\left[\frac{a_1|\Omega|}{\|u_j\|_{X_0}^2}+\frac{a_2'C^{1-1/q}}{\|u_j\|_{X_0}^{2/q}}\right]\,,
\end{aligned}$$
thanks to \eqref{q2}. Thus, \eqref{fPnorma2} follows from this, \eqref{condlemma2} and \eqref{ujtoinfty}.

Finally, dividing both sides of \eqref{sec} by $\|u_j\|_{X_0}^2$, using \eqref{condlemma1}--\eqref{ujtoinfty} and \eqref{fPnorma2} we get
$$\frac{\displaystyle{2\int_\Omega F(x,u_j(x))\, dx -\int_\Omega f(x,u_j(x))u_j(x)\,dx}}{\|u_j\|_{X_0}^2} \to 0\,,$$
which, arguing as above, yields
$$\frac{\displaystyle{2\int_\Omega F(x,u_j(x))\, dx}}{\|u_j\|_{X_0}^2} \to 0\,,$$
as $j\to +\infty$. Of course, this is in contradiction with \eqref{F2}. The proof of Lemma~\ref{lemmanabla2} is complete.
\end{proof}

As a consequence of Lemma~\ref{lemmanabla1} and Lemma~\ref{lemmanabla2}, we get the following result on the validity of the $\nabla$-condition for $\mathcal J_\lambda$.
\begin{proposition}\label{propnabla}
Let $k$ and $m$ in $\NN$ be such that $\lambda_{k-1}<\lambda_k=\dots = \lambda_{k+m-1}<\lambda_{k+m}$ and let $f$ be a function satisfying conditions~\eqref{caratheodory}--\eqref{mu0}.

Then, for any $\sigma>0$ there exists $\varepsilon_\sigma>0$ such that for any $\lambda\in [\lambda_{k-1}+\sigma, \lambda_{k+m}-\sigma]$ and for any $\varepsilon', \varepsilon''\in (0, \varepsilon_\sigma)$, with $\varepsilon'<\varepsilon''$, functional~$\mathcal J_\lambda$ satisfies
the $(\nabla)(\mathcal J_\lambda, \mathbb H_{k-1}\oplus \mathbb P_{k+m-1}, \varepsilon', \varepsilon'')$ condition.
\end{proposition}
\begin{proof}
Assume by contradiction that there exists $\sigma>0$ such that for every $\ve_0>0$ there exist $\bar \lambda\in [\lambda_{k-1}+\sigma,\lambda_{k+m}-\sigma]$
and $\ve'<\ve''$ in $(0,\ve_0)$ such that
\begin{equation}\label{notnabla}
(\nabla)(\mathcal J_{\bar\lambda}, \mathbb H_{k-1}\oplus \mathbb P_{k+m-1}, \varepsilon', \varepsilon'')\,\,\, \mbox{does not hold}.
\end{equation}
Associated to such a $\sigma$, take $\ve_0>0$ as provided by Lemma~\ref{lemmanabla1}.

By \eqref{notnabla} we can find a sequence $\{u_j\}_{j\in \NN}$ in $X_0$ such that
\begin{equation}\label{ujadd}
\begin{aligned}
& \mathcal J_{\bar\lambda}(u_j)\in [\ve',\ve'']\,\,\,  \mbox{for all}\,\, j\in \NN\,,\\
& dist(u_j,\mathbb H_{k-1}\oplus \mathbb P_{k+m-1})\to 0\\
& P_{\mathbb H_{k-1}\oplus \mathbb P_{k+m-1}}\nabla \mathcal J_{\bar\lambda}(u_j)\to 0\,\,\, \mbox{in}\,\, X_0
\end{aligned}
\end{equation}
as $j\to +\infty$.

By Lemma~\ref{lemmanabla2} we know that $\{u_j\}_{j\in \NN}$ is bounded in $X_0$, and so we can assume that for some $u_\infty\in X_0$
\begin{equation}\label{convujadd}
\begin{aligned}
& u_j\rightharpoonup u_\infty\,\,\, \mbox{in}\,\, X_0\\
& u_j\to u_\infty\,\,\, \mbox{in}\,\, L^\nu(\Omega)\,\,\, \mbox{for any}\,\, \nu\in[1, 2^*)\\
& u_j\to u_\infty\,\,\, \mbox{a.e. in}\,\, \Omega
\end{aligned}
\end{equation}
as $j\to +\infty$ and for any $\nu\in [1, 2^*)$ there exists $\ell_\nu\in L^\nu(\RR^n)$ such that
\begin{equation}\label{dominata20addbis}
|u_j(x)|\leq \ell_\nu(x) \quad \mbox{a.e. in}\,\, \RR^n\,\quad \mbox{for any}\,\,j\in \NN\,.
\end{equation}

Now, note that by \eqref{grad} we can write
\begin{equation}\label{add55}
\begin{aligned}
P_{\mathbb H_{k-1}\oplus \mathbb P_{k+m-1}} \nabla \mathcal J_{\bar\lambda}(u_j) & =u_j
-P_{\footnotesize{\mbox{span}}\left\{e_k,\ldots,\,e_{k+m-1}\right\}}u_j\\
& \qquad \qquad \qquad -P_{\mathbb H_{k-1}\oplus \mathbb P_{k+m-1}}\mathcal L_K^{-1}(\bar\lambda u_j+f(x,u_j)).
\end{aligned}
\end{equation}
Hence, recalling that $\mathcal L_K^{-1}:L^{q'}(\Omega)\to X_0$ is a
compact operator (see  Section~\ref{secgrad}), and that $f(x,u_j)\to f(x,u_\infty)$ in $L^{q'}(\Omega)$ by \eqref{caratheodory}-\eqref{crescita} and \eqref{dominata20addbis}, we get that
$$P_{\mathbb H_{k-1}\oplus \mathbb P_{k+m-1}}\mathcal L_K^{-1}(\bar \lambda
u_j+f(x,u_j))\to P_{\mathbb H_{k-1}\oplus \mathbb P_{k+m-1}}\mathcal L_K^{-1}(\bar \lambda
u_\infty+f(x,u_\infty))$$ as $j\to +\infty$ and so, taking into account \eqref{ujadd}, \eqref{convujadd} and \eqref{add55}, we deduce that
\begin{equation}\label{ujforte}
u_j\to P_{\mathbb H_{k-1}\oplus \mathbb P_{k+m-1}}\mathcal L_K^{-1}(\bar \lambda
u_\infty+f(x,u_\infty))=:u_\infty\,\,\, \mbox{in}\,\, X_0
\end{equation}
as $j\to +\infty$.

Furthermore, again by \eqref{ujadd} we have that
$$\langle \nabla \mathcal J_{\bar \lambda}(u_j), \varphi \rangle_{X_0} \to 0\,\,\, \mbox{for any}\,\, \varphi \in \mathbb H_{k-1}\oplus \mathbb P_{k+m-1}\,,$$
that is, taking into account \eqref{gradJ},
\begin{equation}\label{ujcritico}
\langle \mathcal J'_{\bar\lambda}(u_j), \varphi\rangle= \langle u_j, \varphi\rangle_{X_0}
-\lambda \int_\Omega u_j(x)\varphi(x)\,dx-\int_\Omega f(x, u_j(x))\varphi(x)\,dx \to 0
\end{equation}
as $j\to +\infty$. Thus, by \eqref{crescita}, \eqref{convujadd}, \eqref{ujforte} and \eqref{ujcritico}, we obtain that
$$\langle \mathcal J'_{\bar\lambda}(u_\infty), \varphi\rangle= \langle u_\infty, \varphi\rangle_{X_0}
-\lambda \int_\Omega u_\infty(x)\varphi(x)\,dx-\int_\Omega f(x, u_\infty(x))\varphi(x)\,dx$$
for any $\varphi \in \mathbb H_{k-1}\oplus \mathbb P_{k+m-1}$, i.e. $u_\infty$
is a critical point of $\mathcal J_{\bar\lambda}$ constrained on $\mathbb H_{k-1}\oplus \mathbb P_{k+m-1}$.

Hence, Lemma~\ref{lemmanabla1} yields that $u_\infty\equiv 0$. However, $0<\ve'\leq \mathcal J_{\bar\lambda}(u_j)$ for every
$j\in \NN$, so that, by continuity of $\mathcal J_{\bar\lambda}$, we find $\mathcal J_{\bar\lambda}(u_\infty)>0$, which is absurd. This completes the proof of Proposition~\ref{propnabla}.
\end{proof}

\section{Proof of main theorem}\label{sec:proofthmain}
This section is devoted to the proof of main result of the paper, concerning the existence of multiple solutions for problem~\eqref{problemaK}. In order to get this result we apply the following abstract critical point theorem (\cite[Theorem~2.10]{marsac:sns}):
\begin{theorem}[Sphere-torus linking with mixed type assumptions]\label{thmarinosaccon}
Let $H$ be a Hilbert space and $X_1, X_2, X_3$ be three subspaces of $H$ such that $H=X_1\oplus X_2 \oplus X_3$ with $0<\mbox{dim}\,X_i<\infty$ for $i=1,2$\,. Let $\mathcal I:H\to \RR$ be a $C^{1,1}$ functional. Let $\rho, \rho', \rho'', \rho_1$ be such that $0<\rho_1$, $0\leq \rho'<\rho<\rho''$ and
$$\Delta=\{u\in X_1\oplus X_2 : \rho'\leq \|P_2u\|\leq \rho'', \|P_1u\|\leq \rho_1\}\,\,\, \mbox{and}\,\,\,T=\partial_{X_1\oplus X_2}\Delta\,,$$
where $P_i:H\to X_i$ is the orthogonal projection of $H$ onto $X_i$\,, $i=1,2$\,, and
$$S_{23}(\rho)=\{u\in X_2\oplus X_3 : \|u\|=\rho\}\,\,\, \mbox{and}\,\,\,B_{23}(\rho)=\{u\in X_2\oplus X_3 : \|u\|<\rho\}\,.$$
Assume that
$$a'=\sup \mathcal I(T)<\inf \mathcal I(S_{23}(\rho))=a''\,.$$
Let $a,b$ be such that $a'<a<a''$, $b>\sup \mathcal I(\Delta)$ and
$$\mbox{the assumption}\,\,\, (\nabla)(\mathcal I, X_1\oplus X_3, a, b)\,\,\, \mbox{holds};$$
$$\mbox{the Palais--Smale condition holds at any level}\,\,\, c\in [a,b]\,.$$
Then, $\mathcal I$ has at least two critical points in $\mathcal I^{-1}([a,b])$\,.

If, furthermore,
$$a_1<\inf \mathcal I(B_{23}(\rho))>-\infty$$
and the Palais--Smale condition holds at every $c\in [a_1, b]$, then $\mathcal I$ has another critical level between $a_1$ and $a'$\,.
\end{theorem}

First of all, we need the following result:
\begin{lemma}\label{limite}
Let $k$ and $m$ in $\NN$ be such that $\lambda_{k-1}<\lambda<\lambda_k=\dots = \lambda_{k+m-1}<\lambda_{k+m}$ and let $f$ satisfy \eqref{caratheodory}--\eqref{mu0}.

Then, the following relation is verified:
\[
\lim_{\lambda \to \lambda_k}\sup_{u\in \mathbb H_{k+m-1}} \mathcal J_\lambda(u)=0.
\]
\end{lemma}
\begin{proof}
First of all, note that $\mathcal J_\lambda$ attains a maximum in $\mathbb H_{k+m-1}$ by \eqref{a3a4}.

Now, assume by contradiction that there exist $\{\mu_j\}_{j\in \NN}$, such that
\begin{equation}\label{muj}
\mu_j\to \lambda_k
\end{equation}
as $j\to +\infty$, $\{u_j\}_{j\in \NN}$
in $\mathbb H_{k+m-1}$ and $\ve>0$ such that for any $j\in \NN$
\begin{equation}\label{sup}
\mathcal J_{\mu_j}(u_j)=\sup_{u\in \mathbb H_{k+m-1}} \mathcal J_{\mu_j} (u)\geq \ve.
\end{equation}

If $\{u_j\}_{j\in \NN}$ is bounded in $X_0$, we can assume that $u_j\to u_\infty$ in $X_0$ as $j\to +\infty$ for some $u_\infty \in \mathbb H_{k+m-1}$. Then, by \eqref{cond22F}, \eqref{muj} and the fact that $\mathbb H_{k+m-1}$ is finite-dimensional, we have that
$$\mathcal J_{\mu_j}(u_j)\to \mathcal J_{\lambda_k}(u_\infty)$$
as $j\to +\infty$, and so by \eqref{mu0}, \eqref{antipoincare} and \eqref{sup}, we immediately get
\[
\begin{aligned}
\ve\leq \mathcal J_{\lambda_k}(u_\infty) & =\frac{1}{2}\|u_\infty\|_{X_0}^2-\frac{\lambda_k}{2}\int_\Omega
|u_\infty(x)|^2\, dx -\int_\Omega F(x,u_\infty(x))\, dx\\
& \leq \frac 1 2 (\lambda_{k+m-1}-\lambda_k)\int_\Omega
|u_\infty(x)|^2\, dx -\int_\Omega F(x,u_\infty(x))\, dx\leq 0,
\end{aligned}
\]
and a contradiction arises.

Otherwise, if $\{u_j\}_{j\in\NN}$ is unbounded in $X_0$, we can suppose that
\begin{equation}\label{ujadd66}
\|u_j\|_{X_0}\to +\infty
\end{equation}
as $j\to +\infty$. Therefore, \eqref{sup} and \eqref{a3a4} imply
\[
0<\ve\leq \mathcal J_{\mu_j}(u_j)\leq \frac{1}{2}\|u_j\|_{X_0}^2-
\frac{\mu_j}{2}\int_\Omega |u_j(x)|^2\, dx-a_3\int_\Omega |u_j(x)|^q dx+a_4|\Omega|.
\]
But all norms are equivalent in $\mathbb H_{k+m-1}$, so that the right hand side of the previous inequality tends to $-\infty$ as $j\to +\infty$, since $q>2$ by assumption, and \eqref{ujadd66} holds true. Hence, a contradiction arises as well.
\end{proof}

Now, we can prove our multiplicity result for problem~\eqref{problemaK}. The idea consists in applying Theorem~\ref{thmarinosaccon} to $\mathcal J_\lambda$, in connection with a classical Linking Theorem (see \cite[Theorem~5.3]{rabinowitz}). First we prove:
\begin{proposition}\label{2soluzioni}
Let $k$ and $m$ in $\NN$ be such that $\lambda_{k-1}<\lambda<\lambda_k=\dots = \lambda_{k+m-1}<\lambda_{k+m}$ and let $f$ satisfy \eqref{caratheodory}--\eqref{mu0}.

Then, there exists a left neighborhood $\mathcal O_k$ of $\lambda_k$ such that for all $\lambda \in \mathcal O_k$, problem~\eqref{problemaK} has two nontrivial solutions $u_i$ such that
$$0<\mathcal J_\lambda(u_i)\leq \sup_{u\in \mathbb H_{k+m-1}} \mathcal J_\lambda(u)$$
for $i=1,2$.
\end{proposition}
\begin{proof}
The strategy consists in applying Theorem~\ref{thmarinosaccon} to the functional~$\mathcal J_\lambda$. For this purpose, fix $\sigma>0$ and find $\ve_\sigma$ as in Proposition~\ref{propnabla}.
Then, for all $\lambda\in [\lambda_{k-1}+\sigma, \lambda_{k+m}-\sigma]$ and for every $\ve',\ve''\in(0,\ve_\sigma)$, functional~$\mathcal J_\lambda$ satisfies the $(\nabla)(\mathcal J_\lambda, \mathbb H_{k-1}\oplus \mathbb P_{k+m-1}, \varepsilon', \varepsilon'')$ condition.

By Lemma~\ref{limite} there exists $\sigma_1\leq
\sigma$ such that, if $\lambda \in (\lambda_k-\sigma_1,\lambda_k)$, then
\begin{equation}\label{epsilonsup}
\sup_{u\in \mathbb H_{k+m-1}} \mathcal J_\lambda(u)<\ve''\,.
\end{equation}
Moreover,
since $\lambda<\lambda_k$, Proposition~\ref{propgeometria} holds true. Also, $\mathcal J_\lambda$ satisfies the Palais--Smale condition at any level, by Proposition~\ref{lemmaPS}.

Then, by Theorem~\ref{thmarinosaccon}, there exist two critical points
$u_1,u_2$ of $\mathcal J_\lambda$ with
\begin{equation}\label{epsilon}
\mathcal J_\lambda(u_i)\in [\ve',\ve'']\,,
\end{equation}
$i=1,2$. In particular $u_1$ and $u_2$ are non-trivial solutions of
problem~\eqref{problemaK} such that
$$0<\mathcal J_\lambda(u_i)\leq \sup_{u\in \mathbb H_{k+m-1}} \mathcal J_\lambda(u)\,\,\,\,\, i=1,2\,,$$
since $\varepsilon''$ is arbitrary in \eqref{epsilonsup} and \eqref{epsilon}, and this ends the proof of Proposition~\ref{2soluzioni}.
\end{proof}

We are now ready to conclude with the
\begin{proof}[Proof of~Theorem~$\ref{thmain}$]
By the classical Linking Theorem (see \cite[Theorem~5.3]{rabinowitz}), for any $\lambda\in (\lambda_{k-1}, \lambda_k)$ one can prove the existence of a solution $u_3$ of problem~\eqref{problemaK} with
\begin{equation}\label{stimasup}
\mathcal J_\lambda(u_3)\geq \inf_{u\in \mathbb P_{k-1},\, \|u\|=\varrho} \mathcal J_\lambda(u)\geq \beta,
\end{equation}
for suitable $\varrho>0$ and $\beta>0$, see \cite{svlinking}.

By Lemma~\ref{limite}, we can choose $\lambda$ so close to $\lambda_k$ that
\begin{equation}\label{confr}
\sup_{u\in \mathbb H_{k+m-1}} \mathcal J_\lambda (u)< \inf_{u\in \mathbb P_{k-1},\, \|u\|=\varrho} \mathcal J_\lambda(u)\,.
\end{equation}

Hence, inequalities \eqref{stimasup}, \eqref{confr} and Proposition~\ref{2soluzioni} immediately imply that
$$\mathcal J_\lambda(u_i)\leq \sup_{u\in \mathbb H_{k+m-1}} \mathcal J_\lambda(u)<\mathcal J_\lambda(u_3)$$
and so $u_3\neq u_i,$ $i=1,2$. The proof of Theorem~\ref{thmain} is complete.
\end{proof}

\end{document}